\documentclass[12pt,reqno]{amsart}
\usepackage{amsmath,amssymb,amsfonts,amscd,latexsym,amsthm,mathrsfs,verbatim}
\usepackage[unicode]{hyperref}
\newtheorem*{theorem}{Theorem}
\theoremstyle{remark}

\renewcommand{\d}{{\mathrm d}}

\newcommand{\lcm}{\operatorname{lcm}}
\newcommand{\dilog}{\operatorname{dilog}}

\begin{document}

\title[Automatic Discovery of Irrationality Proofs]{Automatic Discovery of Irrationality Proofs and Irrationality Measures}

\date{20 December 2019}

\author{Doron Zeilberger}
\address{Department of Mathematics, Rutgers University (New Brunswick), Hill Center-Busch Campus, 110 Frelinghuysen Rd., Piscataway, NJ 08854-8019, USA}
\email{doronzeil@gmail.com}

\author{Wadim Zudilin}
\address{Department of Mathematics, IMAPP, Radboud University, PO Box 9010, 6500~GL Nijmegen, Netherlands}
\email{w.zudilin@math.ru.nl}


\dedicatory{Dedicated to Bruce Berndt (b.\ March 13, 1939), on his $80\frac{3}{4}$-th birthday}

\begin{abstract}
We illustrate the power of {\it Experimental Mathematics} and {\it Symbolic Computation} to
suggest irrationality proofs of natural constants, and the determination of their irrationality measures.
Sometimes such proofs can be fully automated, but sometimes there is still need for a human touch.
\end{abstract}

\maketitle

\subsection*{The Maple packages}
This article is accompanied by  {\bf four} Maple packages:
\begin{itemize}
\item[$\bullet$] \texttt{ALLADI.txt}: a Maple package, inspired by the article \cite{AR80}.
It does an automated redux of Theorem 1 in their paper, and extends their results to proving irrationality, and
finding irrationality measures, of constants of the form $\int_0^1 \, 1/P(x) \, \d x$, where $P(x)$ is a
quadratic polynomial with integer coefficients.
\item[$\bullet$] \texttt{GAT.txt}: a Maple package that includes the former case, but generalizes it to integrals of the form 
$\int_0^1 \, \frac{1}{a+x^k} \, \d x $, where $a$ and $k$ are positive integers.
\item[$\bullet$] \texttt{BEUKERS.txt}: a Maple package for getting integer linear combinations of $1$, $\operatorname{DiLog}(a/(a-1)))$, and $\operatorname{Log}(a/(a-1))$ that are very small,
for integers $a \geq 2$.
\item[$\bullet$] \texttt{SALIKHOV.txt}: a Maple package that generalizes V. Kh. Salikhov's method \cite{Sa07} 
to discover and prove linear independence measure of $\{1, \log(a/(a+\nobreak1)),\allowbreak \log(b/(b+1)) \}$ for most pairs of integers $2 \leq  a < b$,
in particular for $\{1, \allowbreak \log(a/(a+1)), \log((a+1)/(a+2)) \}$ for all $a \geq 1$.
\end{itemize}
They are available, along with numerous output files, in the form of computer-generated articles, from the following url:\\
\url{http://www.math.rutgers.edu/~zeilberg/mamarim/mamarimhtml/gat.html}.

\subsection*{Preface: Roger Ap\'ery's Astounding Proof (and Bruce Berndt's Seminar Talk) and the Almkvist--Zeilberger algorithm}
In 1978, 64-year-old Roger Ap\'ery, announced, and sketched, an amazing proof
that $\zeta(3)=\sum_{n=1}^{\infty} \frac{1}{n^3}$ is an irrational number.
Some of the details were filled in by Henri Cohen and Don Zagier, and the completed proof
was the subject of Alf van der Poorten's classic~\cite{vdP79}.

One of us (DZ) first learned about this proof from an excellent talk by
Bruce Berndt, delivered at the University of Illinois, way back in Fall 1979
(when DZ had his third postdoc).
Hence it is appropriate that we dedicate the present paper to Bruce Berndt, since
it deals with irrationality of constants inspired by Ap\'ery's seminal proof,
exposited so lucidly by Bruce Berndt. 

Another {\it leitmotif} of the present paper is the Almkvist--Zeilberger algorithm. 
Gert Almkvist and Bruce Berndt co-authored a classic expository paper \cite{AB88} that won a
prestigious MAA Lester Ford award in 1988. Since Gert Almkvist (1934--2018) was also a good friend,
and long-time collaborator, of the second-named author (WZ), and both authors are good friends
and admirers of Bruce Berndt, it is more than fitting to dedicate this article to Bruce Berndt.

\subsection*{Ap\'ery's proof of the Irrationality of $\zeta(3)$}
Roger Ap\'ery (see \cite{vdP79}) {\it pulled out of a hat} two  {\bf explicit} sequences of rational numbers. 
The first sequence consisted only of integers,
$$
b_n \, := \, \sum_{k=0}^{n} {{n} \choose {k}}^2 \,{{n+k} \choose {k}}^2 ,
$$
while the second one was a sequence of rational numbers,
$$
a_n \, := \, \sum_{k=0}^{n} {{n} \choose {k}}^2 \,{{n+k} \choose {k}}^2 \left ( 
\sum_{m=1}^{n}  \frac{1}{m^3} \, + \,  \sum_{m=1}^{k} \frac{(-1)^{m-1}}{2\,m^3\,{{n} \choose {m}} {{n+m}\choose {m}} }
\right ).
$$
It was easy to check that $p_n:=\lcm(1,\dots,n)^3 a_n$ are integers and, of course,
$q_n:=\lcm(1,\dots, n)^3 b_n$ are integers. Then he came up with
a real number $\delta=\frac{4\,\log(1+\sqrt{2}) -3}{4\,\log(1+\sqrt{2}) +3}=0.080259\hdots>0$ such that
for some constant $C$ (independent of $n$)
$$
\biggl|\zeta(3) - \frac{p_n}{q_n}\biggr| \, \leq \, \frac{C}{q_n^{1+\delta}} .
$$
Once all the claims are verified, it follows that $\zeta(3)$ is irrational.
Indeed, if $\zeta(3)$ would have been rational with denominator $c$, the left side would have been bounded below by $\frac{1}{cq_n}$.
It also follows that an {\it irrationality measure} (see \cite{vdP79}) is bounded above by $1+\frac{1}{\delta}=12.417820\dots$\,.

\subsection*{Frits Beukers' Version}

Shortly after, Frits Beukers \cite{Be79} gave a much simpler rendition of Ap\'ery's construction by introducing a certain explicit triple integral
$$
I(n)= \int_0^1 \!\!  \int_0^1 \!\!  \int_0^1
\left ( \frac{x(1-x)y(1-y)z(1-z)}{1-(1-xy)z } \right )^n \frac{\d x \, \d y \, \d z }{1-(1-xy)z},
$$
and pointing out that
\begin{itemize}
\item[(i)] $I(n)$ is small and can be explicitly bounded: $I(n)=O((\sqrt2-1)^{4n})$;
\item[(ii)] $I(n)=A(n) +B(n) \zeta(3)$ for certain sequences of rational numbers $A(n)$, $B(n)$, that can be explicitly bounded; and
\item[(iii)] $A(n)\,\lcm(1, \dots, n)^3$ and  $B(n)\, \lcm(1, \dots, n)^3$ are
integers. 
\end{itemize}
Since, thanks to the Prime Number Theorem $\lcm(1,\dots, n)$ is $\Omega(e^n)$, everything followed.

(We use the convenient notation $F(n)=\Omega(c^n)$ meaning $\lim_{n \rightarrow \infty} (\log F(n))/n \nobreak= c $.)

Shortly after, Krishna Alladi and M.\,L.~Robinson, used one-dimensional analogs to reprove the irrationality of $\log 2$,
and established an upper bound for its irrationality measure of $4.63$ (subsequently improved, see \cite{Wei}) by considering the simple integral
$$
I(n)= \int_0^1 \left ( \frac{x(1-x)}{1+x} \right )^n \, \frac{\d x}{1+x} .
$$

\subsection*{Man muss immer umkehren}

Carl Gustav Jacob Jacobi told us that {\it one must always invert}. Of course, he meant that if you have a complicated
looking function like $\int_{0}^{x} \frac{1}{\sqrt{1-t^2}} \, \d t$ its inverse function, in this case $\sin x$,
is much more user-friendly. We  understand Jacobi's quip in a different way.
Rather than start with a famous constant, say $\zeta(3)$ or $\log 2$, and wreck
our brains trying to find Beukers-like or Alladi--Robinson-like integrals that would produce good diophantine
approximations, start with a family of integrals $I(n)$,
and see 
\begin{itemize}
\item[$\bullet$] whether $I(0)$ is a familiar constant, let's call it $x$;
\item[$\bullet$] whether $I(n)$, for integers  $n>0$, can be written as $A(n)-B(n)x$, with $\{A(n)\}$ and $\{B(n)\}$ sequences of rational numbers;
\item[$\bullet$] whether $I(n)$ has exponential decay, i.e.\ is `small';
\end{itemize}
Since $A(n)-B(n)x$ is so close to $0$, $\frac{A(n)}{B(n)}$ is very close to $x$. 
Write $\frac{A(n)}{B(n)}$ as  $\frac{A'(n)}{B'(n)}$, where now $A'(n)$ and $B'(n)$ are \emph{integers}.
How close is $A(n)/B(n)=A'(n)/B'(n)$ to~$x$, from a {\it diophantine} (rather than numerical-analysis) point of view?

We are looking for what we call `empirical delta', let's call it $\delta(n)$, such that
$$
\biggl|x- \frac{A'(n)}{B'(n)}\biggr| = \frac{1}{B'(n)^{1+\delta(n)}} ;
$$
in other words, we define
$$
\delta(n) := \frac{-\log \bigl|x- \frac{A'(n)}{B'(n)}\bigr|}{\log B'(n)} -1 .
$$
If the values of $\delta(n)$ for, say $990 \leq n \leq 1000$ are all strictly positive, and safely not {\it too} close to $0$, then we
can be assured that there {\it exists} a proof of irrationality, and a corresponding rigorous upper bound for its irrationality measure.

But being sure that a proof {\it exists} is not the same as having one. It would be nice to have a fully rigorous proof.
First, try to find one yourself, and you are welcome to get help from your computer, that excels not only in
{\it number-crunching}, but also in {\it symbol-crunching}, but is still not so good in {\it idea-crunching}
(although it is getting better and better!).

If you are stuck, you can always email an expert number-theorist and offer him or her to collaborate with you on the paper\\
\centerline{``Proof of the irrationality of $x$".}\\
If $x$ was not yet proved to be irrational,
and furthermore is sufficiently famous (e.g.\ Euler's constant, $\gamma$, or Catalan's constant, $G$) you and your collaborator
would be famous too (that what happened to Ap\'ery). If the constant in question is complicated and obscure,
it is still publishable, at least in the \href{http://arxiv.org}{\texttt{arXiv}}.
If $x$ is already proved to be irrational, and there is currently a proved upper bound for
the irrationality measure of $x$ and the implied (rigorous) bound from your sequence $A(n)/B(n)$  is better
(i.e.\ smaller) than the previous one, you (and your expert collaborator) can write a paper\\
\centerline{``A new upper bound for the irrationality measure of $x$",}\\
and you (and your expert collaborator) would be known as the current holders of the world record of the irrationality measure of $x$,
until someone else, in turn, would break your record.

In this paper we will show how, using the amazing Almkvist--Zeilberger algorithm \cite{AZ90} that finds (and at the same time, proves!)
a linear recurrence equation with polynomial coefficients for such integrals $I(n)$, one can accomplish the first four steps (that we call {\it reconnaissance})
very fast and very efficiently, and sometimes (but not always!) the last `rigorous', step, can also be
automated.

\subsection*{Warm Up: Computerized irrationality proof of $\log 2$}

Consider the sequence of definite integrals
$$
I(n) := \int_0^1 \left (\frac{x(1-x)}{1+x}\right )^n \, \frac{\d x}{1+x} .
$$

The {\bf Almkvist--Zeilberger} algorithm \cite{AZ90} 
produces a linear recurrence equation with polynomial coefficients satisfied by $I(n)$:
$$
\left( n+2 \right) I \left( n \right) + \left( -6\,n-9 \right) I \left( n+1 \right) + \left( n+1 \right) I \left( n+2 \right) \, = \, 0 .
$$
From this we can compute, very fast, many values, and find out the `empirical deltas'. For example,
\begin{align*}
I(50)
&=
-{\frac{1827083538922494024488153994990786998947102154393958429773}{172169139124777594800}}
\\ &\qquad
+15310086199495855930932559804210504653\,\ln  \left( 2 \right) \quad.
\end{align*}
This implies the rational approximation to $\log 2$  (by `pretending' that $I(50)$ is zero) of
$$
\frac{1827083538922494024488153994990786998947102154393958429773}{2635924360893339481850468164186010894239049450495548604400} ,
$$
whose empirical delta is
$$
 0.33269846131126944438\dots \quad .
$$
This is encouraging!  But we can't judge  from just one data point. We next find that
the `empirical deltas' for $n=51$ and $n=53$ are $0.31992792581569268673\dots$ and  $0.30031107795443952791\dots$, respectively.
To get more confidence, we need to go to higher values of $n$.
The lowest empirical delta between $n=990$ and $n=1000$ turns out to be
$$
0.28193333613008344616\dots \,  ,
$$
that is not as good, but is way above $0$.
It leads to an estimate for the irrationality measure of $4.5469377751717949058$\,. This trend persists, so we can be
convinced that the integral $I(n)$ is promising. But this is {\bf not} yet a rigorous proof.

Writing  $I(n)=A(n)+B(n) \log(2)$, the next step is to (automatically!) find the rate of growth of $A(n)$ and $B(n)$.
The original proof in \cite{AR80} used partial fractions, and  the saddle-point method, but thanks to the
Poincar\'e lemma (see \cite{vdP79}) we can do it very fast and automatically.

Note that $A(n)$ and $B(n)$ satisfy the same recurrence. In other words,
\begin{gather*}
\left( n+2 \right) A \left( n \right) + \left( -6\,n-9 \right) A \left( n+1 \right) + \left( n+1 \right) A \left( n+2 \right) \, = \, 0 ,
\\
\left( n+2 \right) B \left( n \right) + \left( -6\,n-9 \right) B \left( n+1 \right) + \left( n+1 \right) B \left( n+2 \right) \, = \, 0 .
\end{gather*}
The `constant-coefficient approximation' of the above recurrence is (taking the leading coefficient in $n$, that happens to be $n^1=n$)
$$
 \bar{A} \left( n \right) -6 \bar{A} \left( n+1 \right) +  \bar{A} \left( n+2 \right) \, = \, 0 \quad ,
$$
where $\bar{A}(n)$ is an approximation to $A(n)$ that Poincar\'e proved has the property that 
$$
\lim_{n \rightarrow \infty} \frac{\log \bar{A}(n)}{\log A(n)} \, = \, 1 ,
$$
and similarly for $B(n)$ and $\bar{B}(n)$. The {\it indicial} equation of this constant-coefficient linear recurrence is
$$
1-6N+N^2 =0 ,
$$
whose roots are
$$
a=3+2\,\sqrt {2} , \quad b=3-2\,\sqrt {2} .
$$
It follows that $|A(n)|,|B(n)|=\Omega(a^n)$  and that $|I(n)|=\Omega(b^n)$, since $I(n)$ is obviously the sub-dominant solution, of exponential {\bf decay}. 
There is only one problem, $A(n)$ is not an integer. The computer can easily check, empirically that
$A'(n):=\lcm(1,\dots,n)A(n)$ is an integer for $1 \leq n \leq 1000$, and then try to prove it in general (or get a little help
from a human friend). Then defining
$$
A'(n)= \lcm(1, \dots ,n) A(n) , \quad
B'(n)= \lcm(1,\dots ,n) B(n) , \quad
I'(n)= \lcm(1,\dots, n) I(n) ,
$$
we have that $A'(n),B'(n)$ are integers. By the  Prime Number Theorem, $\lcm(1, \dots, n)=\Omega(e^n)$, hence
$$
|A'(n)|, \; B'(n) = \Omega(a^n \, e^n) , \quad  |I'(n)|= \Omega(b^n \, e^n) .
$$
Since we want
$$
|I'(n)|= \frac{1}{\Omega(B'(n)^{\delta})} ,
$$
we can take
$$
\delta = \frac{-\log (be)}{\log(ae)}\, =\,
-\frac{\log(b)+1}{\log(a)+1}  = 
\frac{\log  \left( 3+2\,\sqrt {2} \right) -1}
{\log  \left( 3+2\,\sqrt {2} \right) +1} ,
$$
leading to the Alladi--Robinson upper bound  of $1+1/\delta$ that equals
$$
2\,{\frac {\ln  \left( 3+2\,\sqrt {2} \right) }{\ln  \left( 3+2\,\sqrt {2} \right) -1}} \, = \,
 4.6221008324542313342\dots \,.
$$
This has been improved several times \cite{Wei}, first by Ekaterina Rukhadze \cite{Ru87} (see also \cite{Zu04}); the current record of $3.57455391$ is due to Raffaele Marcovecchio \cite{Ma09}.

The novelty of our approach is that it can be taught to a computer, and everything, except possibly proving the
`divisibility lemma' (that in this case is extremely simple, but in other cases is not so simple).

Using this method, our computer, Shalosh B.~Ekhad, proved {\it ab initio}, {\it all by itself} (except the
simple divisibility lemma) Theorem~1 of \cite{AR80}. Note that nowhere did we mention {\it Legendre polynomials}
(they turned, in hindsight,  to be unnecessary). Furthermore, our approach is {\it streamlined}, and
the formulation of the theorem is more explicit.

\begin{theorem}[Alladi--Robinson \cite{AR80}, but with a more explicit formulation]
Let $a$ and $b$ be positive integers such that $a>(b- e^{-1})^2/4$, then $\log(1+b/a)$ is an irrational number
with an irrationality measure that is at most  

$$
\frac {\ln  \bigl( 2\,a+b-2\,\sqrt {a ( a+b ) }\, \bigr) -\ln  \bigl( 2\,a+b+2\,\sqrt {a ( a+b ) } \,\bigr) }
{\ln  \bigl( 2\,a+b-2\,\sqrt {a ( a+b ) } \,\bigr) +1}
$$
\end{theorem}

\begin{proof}[Computer-generated proof]
See\\
\url{http://sites.math.rutgers.edu/~zeilberg/tokhniot/oALLADI1.txt}.
\end{proof}

The above computer-generated paper was produced with the Maple package\\ \texttt{ALLADI.txt}.
We now mention  other articles generated by this Maple package.

\smallskip\noindent
$\bullet$ If you want to see computer-generated proofs of irrationality of $89$  different constants, and possibly new
irrationality measures for each of them, you are welcome to read \\
\url{http://sites.math.rutgers.edu/~zeilberg/tokhniot/oALLADI2.txt} . \\
They were ``cherry-picked" from the `candidate pool' of $\int_0^1 1/(a+bx+cx^2)\,\d x$ with $1 \leq a,b,c \leq 10$, 
with $\gcd(a,b,c)=1$ that consisted of $841$ `applicants', and naturally we only listed our successes.
Of course, all such constants are already proved to be irrational by heavy-artillery theorems
(for example the Lindemann--Weierstrass theorem), but these theorems do not give explicit bounds for the
irrationality measure, and may be {\it ineffective}.

\smallskip\noindent
$\bullet$ Moving right along, the computer-generated article\\
\url{http://sites.math.rutgers.edu/~zeilberg/tokhniot/oALLADI3.txt} \\
gives irrationality proofs, and irrationality measures, to $43$ constants, for integrals of the form $\int_0^1 1/(a+cx^2) \, \d x$
for relatively prime integer pairs $a,c$ in the range $3 \leq a,c \leq 40$. These are probably subsumed in
previous works of Salikhov and his students.

\smallskip\noindent
$\bullet$ Even more impressive is\\
\url{http://sites.math.rutgers.edu/~zeilberg/tokhniot/oALLADI4.txt} , \\
that like the above Alladi--Robinson theorem is true for `infinitely many' constants, i.e.
it is true for symbolic $a$ (subject to the congruence condition).

This theorem may be not new, but the {\it novelty} is that it was completely {\it computer-generated}. Let us
state the theorem proved in that article.

\begin{theorem}
Let $a$  be a positive integer such that $a  \bmod 4$ is  $3$.
    Then   $\arctan(\sqrt{a})/\sqrt{a}$  is an irrational number, with an irrationality measure at most
$$
\frac{\ln \bigl( -a+\sqrt {a(a+1)}\,\bigr) - \ln \bigl( a+\sqrt {a(a+1)}\,\bigr)}
{\ln \bigl( -a+\sqrt {a(a+1)}\,\bigr)  - \ln \sqrt{a} +1} .
$$
\end{theorem}

\subsection*{The Maple packages \texttt{GAT.txt} and \texttt{BEUKERS.txt}}

The Maple package \texttt{GAT.txt} did not produce (so far) anything exciting, so we don't talk about it here, but
the readers are welcome to explore it on their own.

So far, all our integrals were one-dimensional. The Maple package  \texttt{BEUKERS.txt} tries to tweak Frits Beukers' elegant
proof ]B] of the irrationality of $\zeta(2)$, by trying to see what happens when you look at the double integral
$$
E(n,a) := \int_{0}^{1}\!\!\int_{0}^{1} \left ( \frac{x(1-x)y(1-y)}{1-xy/a} \right )^n \, \frac{\d x\,\d y}{1-xy/a} .
$$

The original case of $a=1$ gave an irrationality proof (and measure) for $\zeta(2)=\pi^2/6$, and indeed $E(0,1)=\zeta(2)$.
The recurrence for $E(n,1)$ is second-order (the same one for which Ap\'ery  proved the irrationality of $\zeta(2)$).
However, things get more complicated for higher $a$.

Since $E(0,a)=a\, \dilog((a-1)/a)$ we hoped that considering the above double integral would yield irrationality proofs for
them. Alas, $E(n,a)$ gets `contaminated' with $\log((a-1)/a)$, and it turns out that
$$
E(n,a)=A(n,a)+ B(n,a) \dilog((a-1)/a)+ C(n,a) \log((a-1)/a) \quad ,
$$
for all $n$, for {\bf three}  sequences  of rational numbers $\{A(n,a)\}$, $\{B(n,a)\}$, $\{C(n,a)\}$. This can be proved directly, but
it follows immediately from the fact that it is true for $n=0,1,2$, and that $E(n,a)$ , and hence $A(n,a),B(n,a),C(n,a)$,  satisfy the third-order
linear recurrence equation
\begin{align*}
&
-{a}^{3} \left( n+1 \right) ^{2} \left( a-1 \right)  \left( 32\,na+76\,a-27\,n-66 \right) \, E(n,a)
\displaybreak[2]\\ &\quad
+ \, {a}^{2} ( 512\,{a}^{3}{n}^{3}+2752\,{a}^{3}{n}^{2}-1072\,{a}^{2}{n}^{3}+4800\,{a}^{3}n-5792\,{a}^{2}{n}^{2}
\\ &\quad\qquad
+636\,a{n}^{3}+2736\,{a}^{3}-10140\,{a}^{2}n
+3456\,a{n}^{2}-81\,{n}^{3}-5796\,{a}^{2}
\\ &\quad\qquad
+6068\,na-441\,{n}^{2}+3472\,a-768\,n-432 ) \, E(n+1,a)
\displaybreak[2]\\ &\quad
+ \, a ( 256\,{a}^{2}{n}^{3}+1632\,{a}^{2}{n}^{2}-120\,a{n}^{3}+3376\,{a}^{2}n-780\,a{n}^{2}
-81\,{n}^{3}
\\ &\quad\qquad
+2232\,{a}^{2}-1670\,na-522\,{n}^{2}-1170\,a-1086\,n-717 ) \, E(n+2,a)
\displaybreak[2]\\ &\quad
\, +  \left( 32\,na+44\,a-27\,n-39 \right)  \left( 3+n \right) ^{2} \, E(n+3,a) \, = \, 0 .
\end{align*}
This complicated recurrence was obtained using the Apagodu--Zeilberger {\bf multivariable Almkvist--Zeilberger algorithm} \cite{AZ06}.

Using this recurrence we proved that for every integer $a \geq 2$,
there exists a {\bf positive} $\delta=\delta(a)$ and three sequences of integers
$C_1(n,a)$, $C_2(n,a)$, $C_3(n,a)$ such that
\begin{align*}
&
|C_1(n,a)+ C_2(n,a)\, \dilog((a-1)/a) + C_3(n,a) \log((a-1)/a)|
\\ &\quad
\leq \frac{\text{CONSTANT}}{\max(|C_1(n,a)|,|C_2(n,a)|,|C_3(n,a)|)^{\delta(a)}} .
\end{align*}
The explicit expression for $\delta(a)$ is complicated and we refer the reader to the computer-generated article\\
\url{http://sites.math.rutgers.edu/~zeilberg/tokhniot/oBEUKERS1.txt} ,\\
that contains a fully rigorous proof of this theorem. If we did not know that $\log((a-1)/a)$ was irrational, this
theorem would have implied that $\dilog((a-1)/a)$ and $\log ((a-1)/a))$ can not be both rational.
It is not enough, by itself, to prove the linear independence of $\{1, \log((a-1)/a),  \dilog((a-1)/a)\}$, but
some human modification of it makes the things work well\,---\,see the latest achievements in this direction, together with historical notes, in the wonderful paper \cite{RV19} of Georges Rhin and Carlo Viola.

\subsection*{The Maple package \texttt{SALIKHOV.txt}}

We are most proud of this last Maple package, since it generated a {\bf new} theorem that
should be of interest to `mainstream', human mathematicians. It was obtained by
generalizing V. Kh. Salikhov's proof \cite{Sa07} of the linear independence of $\{1, \log  2, \log 3\}$.
Our computer proved the following theorem.

\begin{theorem}
For any positive integer $a$, the set of three numbers $\{1, \log(a/(a+1)),\allowbreak \log((a+1)/(a+2))\}$ are  linearly independent.
In other words there exists a positive number, $\nu(a)$, such that
 if $q$,$p_1$,$p_2$ are integers and $Q=max(|q|,|p_1|,|p_2|)$ , $Q \geq Q_0$, where $Q_0$ is a sufficiently large number, then
$$
|q \,+ \, p_1\,\log(a/(a+1)) \, + \, p_2\,\log((a+1)/(a+2))| \, > \, \frac{1}{Q^{\nu(a)}} .
$$
\end{theorem}

The full proof is in the following article\\
\url{http://sites.math.rutgers.edu/~zeilberg/tokhniot/oSALIKHOV2.txt} , \\
where an exact expression for $\nu(a)$ can be found (see also below).
The theorem was previously known to be true for $a\ge53$ by Masayoshi Hata \cite{Ha98} 
(recently improved, though somewhat implicitly, by Volodya Lysov to $a\ge32$).

Because of the novelty, we choose this result to feature some human-generated highlights of the proof.
The integrals
$$
E_1(n)=\int_0^{2a+1}\biggl(\frac{x^2(x^2-(2a+1)^2)(x^2-(2a+3)^2)}{(x^2-(2a+1)^2(2a+3)^2)^2}\biggr)^n\frac{\d x}{x^2-(2a+1)^2(2a+3)^2}
$$
and
$$
E_2(n)=\int_0^{2a+3}\biggl(\frac{x^2(x^2-(2a+1)^2)(x^2-(2a+3)^2)}{(x^2-(2a+1)^2(2a+3)^2)^2}\biggr)^n\frac{\d x}{x^2-(2a+1)^2(2a+3)^2}
$$
are generalizations of integrals in Salikhov's article \cite{Sa07} , and we have
$$
E_1(n)=A_1(n)+B(n)\log\frac{a+1}{a+2}
\quad\hbox{and}\quad
E_2(n)=A_2(n)+B(n)\log\frac a{a+1},
$$
where all $E_1(n)$, $E_2(n)$, $A_1(n)$, $A_2(n)$, $B(n)$ satisfy the same third order linear recurrence equation
with polynomial coefficients, 
the indicial polynomial of whose `constant-coefficients recurrence approximation' is
\begin{align*}
&
1+(4a^4+16a^3-11a^2-54a-34)N
\\ &\quad
-(108a^6+648a^5+1440a^4+1440a^3+614a^2+76a-1)N^2+a(a+2)N^3  .
\end{align*}
This polynomial has three real zeroes $C_1(a)$, $C_2(a)$, $C_3(a)$ for $a\ge1$ located as follows:
$$
-\frac1{4a^2(a+2)^2} < C_3(a) < 0 < C_2(a) < \frac1{27a(a+2)} < 108a^2(a+1)^2 < C_1(a) ;
$$
also $C_2(a)>\frac1{4a^2(a+2)^2}>|C_3(a)|$ for $a\ge2$.
Furthermore, choosing $K(a)=a(a+2)$ if $a$ is odd, and $K(a)=(a/2)(a/2+1)$ if $a$ is even, we get numbers $K(a)^nA_1(n)$, $K(a)^nA_2(n)$, and $K(a)^nB(n)$ integral. Then
$$
\nu(a) \le -\frac{\log C_1(a)+\log K(a)+2}{\log C_2(a)+\log K(a)+2}
$$
for $a\ge2$, and the same formula for $a=1$ with $\log C_2(a)$ replaced with $\log|C_3(a)|$.
The upper bound for $\nu(a)$ is asymptotically $\frac{3\log a(a+2)}{\log27-2}$ if $a$~is odd and $\frac{3\log a(a+2)}{\log108-2}$ if $a$~is even, as $a\to\infty$.

\subsection*{Conclusion}
Humans, no matter how smart, can only go so far. Machines, no matter how fast, can also
only go so far. The future of mathematics depends on a fruitful {\it symbiosis} of the strong
points of both species, as we hoped we demonstrated in this modest tribute to Bruce Berndt.


\end{document}